\documentclass[12pt,a4paper]{amsart}

\usepackage{slashed}
\declareslashed{}{/}{0}{0}{\mathfrak{R}}

\newtheorem{thm}{Theorem}[section]
\newtheorem{defn}[thm]{Definition} 
\newtheorem{lem}[thm]{Lemma}

\newtheorem{rem}[thm]{Remark}

\usepackage{hyperref}

\title[optimal set for the quantitative isoperimetric ratio]{A note on existence of an optimal set for a Bonnesen type quantitative isoperimetric ratio in the plane}

\author[Bove]{Silvio Bove}

\author[Croce]{Gisella Croce}

\author[Pisante]{Giovanni Pisante}

\address[S. Bove]{Dipartimento di Matematica e Fisica, Universit\`a degli Studi della Campania "Luigi Vanvitelli", Viale Lincoln 5, 81100 Caserta, Italy}
\email{silvio.bove@studenti.unicampania.it}

\address[G. Croce]{Normandie Univ, UNIHAVRE, LMAH, FR-CNRS-3335, 76600 Le Havre, France}
\email{gisella.croce@univ-lehavre.fr}

\address[G. Pisante]{Dipartimento di Matematica e Fisica, Universit\`a degli Studi della Campania "Luigi Vanvitelli", Viale Lincoln 5, 81100 Caserta, Italy}
\email{giovanni.pisante@unicampania.it}

\begin{document}

\maketitle

\begin{abstract}
In this note we prove the existence of a set $E_0\subset\mathbb{R}^2$, different from a ball, which minimizes, among the convex sets that satisfy a suitable interior cone condition, the ratio \begin{equation}
\label{eq:0}
\frac{D(E)}{\lambda_\mathcal{H}^2(E)},
\end{equation} where $D$ is the
isoperimetric deficit and $\lambda_\mathcal{H}$ the deviation from the spherical shape of a set $E\subset \mathbb{R}^2$.
\end{abstract}

\section{Introduction}
A few years after the Hurwitz proof of the isoperimetric inequality in the plane \cite{zbMATH02662735}, Bernstein in \cite{Bernstein1905} and later Bonnesen in \cite{Bonnesen1924} studied the quantitative problem for planar convex sets. The case of convex sets in any dimension was settled much later by Fuglede in \cite{article} who proved in particular that if $E\subset\mathbb{R}^n$ is a convex set with the same volume of the unit ball $B_1$, then, up to a translation, the Hausdorff distance
from $E$ to $B_1$ is controlled by a suitable power of the difference $P(E)-P(B_1)$. It has be proved in \cite{10.2307/40345432} that for any set of finite perimeter $E\subset\mathbb{R}^n$ with finite measure, it holds
\begin{equation}
\label{eq:1}
\lambda(E)\leq C D(E)^{\frac{1}{2}}
\end{equation}
where $\lambda(E)$ is the so-called Fraenkel asymmetry of $E$ defined as $$\lambda(E)=\min\left\{\frac{|E\Delta B_r(x)|}{r^n}:x\in \mathbb{R}^n,|E|=\omega_nr^n\right\},$$
$C$ depends only on $n$ and $D(E)$ stands for the isoperimetric deficit $$D(E)=\frac{P(E)-P(B_r)}{P(B_r)},$$ where $B_r(x)$ is the ball with radius $r>0$ and center $x$ and $B_r=B_r(0)$.
It has be proved in \cite[Theorem 1.1]{bianchini:hal-01181104} the existence of a set $E_0$ which minimizes the shape functional $$\mathcal{F}(E)=\frac{D(E)}{\lambda^2(E)}$$ among all the subsets of $\mathbb{R}^2$ (the ball excluded).

As in \cite{FUSCO2012616} we introduce a class of sets satisfying an interior cone condition. Given $x\in \mathbb{R}^n$, $R>0$, $\theta\in (0,\pi)$ and $\nu\in \mathbb{S}^{n-1}$, the spherical sector with vertex in $x$, axis of symmetry parallel to $\nu$, radius $R$ and aperture $\theta$ is defined as $$S^{\theta,R}_{x,\nu}=\left\{y\in \mathbb{R}^n:|y-x|<R,\langle y-x,\nu\rangle>\cos(\theta/2)|y-x|\right\}.$$

\begin{defn}
We say that a closed set $E\subseteq\mathbb{R}^n$ satisfies the interior cone condition at the boundary with radius $R>0$ and aperture $\theta$ if for any $x\in \partial E$ there exists $\nu_x\in \mathbb{S}^{n-1}$ such that $S^{\theta,R}_{x,\nu_x}\subset E$.
\end{defn}

\begin{defn} 
Given $R>0$, we denote by $\mathcal{C}_R$ the family of closed sets $E\subseteq\mathbb{R}^n$, with $|E|<\infty$, satisfying the interior cone condition at the boundary with radius $R|E|^{\frac{1}{n}}\omega_n^{-\frac{1}{n}}$ and aperture $\pi/2$. 
\end{defn}
We define the deviation from the spherical shape of a set $E\subset\mathbb{R}^n$ with finite measure as $$\lambda_\mathcal{H}(E)=\min_{z\in \mathbb{R}^n}\left\{\frac{d_\mathcal{H}(E,B_r (z))}{r}:|E|=\omega_n r^n\right\},$$ where $$d_\mathcal{H}(E,B_r(z))=\max\left\{\max_{x\in B_r(z)}dist(x,E),\max_{y\in E}dist(y,B_r(z))\right\}$$ and $B_r(z)$ is the closed ball centered in $z$ with radius $r$.

Let $E_n\subseteq\mathbb{R}^2$ be a sequence of sets and let $E_0\subseteq\mathbb{R}^2$. We say that the sequence $E_n$ converges in the sense of Hausdorff to $E_0$ if
$$d_\mathcal{H}(E_n,E)\rightarrow 0\qquad \textit{when}\ n\to\infty.$$
We will denote this convergence by $E_n\xrightarrow{\mathcal{H}} E_0$.

In \cite{FUSCO2012616} it was proved the following theorem:
\begin{thm}
\label{teo_1.1}
For any $R>0$ there exists $0<\delta_R<1$ and a costant $C=C(R,n)$ depending only on $R$ and $n$ such that, for any $E\in \mathcal{C}_R$ with $D(E)<\delta_R$,
\begin{equation}
\label{eq:3}
\lambda_\mathcal{H}(E)\leq C\begin{cases}D(E)^{\frac{1}{2}}\qquad &\mbox{for}\ n=2\\D(E)^{\frac{1}{2}}(\log \frac{1}{D(E)})^{\frac{1}{2}}\qquad &\mbox{for}\ n=3\\D(E)^{\frac{1}{n-1}}\qquad &\mbox{for}\ n\geq 4.   \end{cases}
\end{equation}
\end{thm}

\begin{rem}
We remark that the family $\mathcal{C}_R$ is scale invariant and that a set in $\mathcal{C}_R^1=\{E\in \mathcal{C}_R:|E|=\omega_n\}$ satisfies the cone condition with radius $R$. If $F\in \mathcal{C}_R$, setting $E=(\omega_n^{1/n}/|F|^{1/n})F$, then $E\in \mathcal{C}_R^1$, $D(E)=D(F)$ and $\lambda_\mathcal{H}(E)=\lambda_\mathcal{H}(F)$.
\end{rem}

We denote with $\mathcal K$ the class of closed convex set of $\mathbb R^n$ and set $\mathcal{K}_R^1=\mathcal K \cap \mathcal{C}_R^1$ the class of the convex sets satisfying the interior cone condition with $|E|=\omega_n$.\\
In this note we prove the existence of a set in the class $\mathcal{K}_R^1$ that minimizes the functional $$\mathcal{F}(E)=\frac{D(E)}{\lambda^2_\mathcal{H}(E)}$$ for $R$ small enough.

\section{Main Result}
Before stating the main theorem let us prove some preliminary results.

\begin{lem}
\label{lem:1}
Let $E_n\subseteq \mathbb{R}^2$ be a sequence of sets such that $E_n\in \mathcal{K}_R^1$. If $E_n\xrightarrow{\mathcal{H}} E_0$, then $E_0\in \mathcal{K}^1_R$.
\end{lem}
\begin{proof}
The convexity of $E_0$ is a consequence of the Hausdorff convergence and it is well known. We give a proof for the convenience of the reader. Let be $x_1$, $x_2\in E_0$. Since $x_1\in E_0$ then exists a sequence $x_n^1$ such that $x_n^1\in E_n$ for all $n$ and $x_n^1\rightarrow x_1$. Similarly exists a sequence $x_n^2$ such that $x_n^2\in E_n$ for all $n$ and $x_n^2\rightarrow x_2$. By choosing $n$ sufficiently large we obtain that simultaneously $x_n^1$ and $x_n^2$ belong to the same $E_n$. By the convexity of $E_n$ we obtain that the segment of extremes $x_n^1$ and $x_n^2$ is contained in $E_n$. Therefore for all $t\in [0,1]$ it results that $x_n^1+t(x_n^2-x_n^1)\in E_n$. By the hypothesis of $E_n\xrightarrow{\mathcal{H}} E_0$ we obtain that $x_1+t(x_2-x_1)\in E_0$ for all $t\in [0,1]$ that is to say that $E_0\in \mathcal{K}$.

In order to prove that $E_0\in \mathcal{C}_R^1$ we need some premises. Let $G$ be a set such that $\partial E_n\xrightarrow{\mathcal{H}} G$. We shall prove that $\partial E_0\subset G$. By contraddiction we suppose that exists $x_0\in \partial E_0$ such that $x_0\notin G$. Therefore exists $\varepsilon>0$ such that $$B(x_0,\varepsilon)\cap \partial E_n=\emptyset\qquad \forall\ n\in \mathbb{N}.$$
Given $E\subseteq\mathbb{R}^n$ we will denote with 
$\mathring{E}$ the interior of $E$ and with $E^c$ its complement. There are two possibilities: 
\begin{enumerate}
\item[1)] $B(x_0,\varepsilon)\subset \mathring{E}_n\qquad \forall n\in \mathbb{N}$;
\item[2)] $B(x_0,\varepsilon)\subset E_n^c\qquad \forall n\in \mathbb{N}$.
\end{enumerate} 
\begin{enumerate}
\item[1)] In the first case we have $B(x_0,\varepsilon/2)\subset \mathring{E}_0$ and so $x_0\in \mathring{E}$, contrary to $x_0\in \partial E_0$.
\item[2)] In the second case, since $x_0\in \partial E_0$ and $E_n\xrightarrow{\mathcal{H}} E_0$, then exists $\overline{x}\in B(x_0,\varepsilon)\cap E_0$ and therefore exists $x_n\in E_n$ for all $n\in \mathbb{N}$ such that $x_n\rightarrow \overline{x}$ contrary $B(x_0,\varepsilon)\subset E_n^c$ for all $n\in \mathbb{N}$.
\end{enumerate}
Therefore $\partial E_0\subset G$.
 
We only need to prove that $E_0\in \mathcal{C}_R^1$. Let be $x_0\in \partial E_0$. Then exists, for all $n\in \mathbb{N}$, $x_n\in \partial E_n$ such that $x_n\rightarrow x_0$. Since $E_n\in \mathcal{C}_R^1$, then exists $S_{x_n,\nu_n}\subset E_n$. Let $\overline{\nu}\in \mathbb{S}^{n-1}$ such that $\nu_n\rightarrow \overline{\nu}$. We claim that $S_{x_0,\overline{\nu}}\subset E_0$. Indeed 
\begin{equation}
\label{eq:6}
\overline{S}_{x_n,\nu_n}\xrightarrow{L^1} \overline{S}_{x_0,\overline{\nu}}.
\end{equation}
If by contradiction we suppose that $S_{x_0,\overline{\nu}}\cap E_0^c\neq \emptyset $ then exists $\tilde{x}\in S_{x_0,\overline{\nu}}\cap E_0^c$ and so exists $B(\tilde{x},\varepsilon)\subset S_{x_0,\overline{\nu}}\cap E_0^c$ but $E_n\xrightarrow{\mathcal{H}} E_0$ and then $B(\tilde{x},\varepsilon)\subset S_{x_0,\overline{\nu}}\cap E_n^c$ for all $n\in \mathbb{N}$ which contradicts \eqref{eq:6} and $S_{x_n,\nu_n}\subset E_n$.
\end{proof}

\begin{lem}
\label{lem:2}
The deviation $\lambda_\mathcal{H}$ from the spherical shape of a set $E\subseteq \mathbb{R}^n$ is upper semi-continuous with respect to the Hausdorff convergence.
\end{lem}
\begin{proof}
Let $E_0\in \mathbb{R}^n$ and $E_n\subseteq\mathbb{R}^n$ be a sequence such that $E_n\xrightarrow{\mathcal{H}} E_0$. Therefore we can write  $$d_\mathcal{H}(E_n,E_0)\rightarrow 0.$$
We denote by $B_n$ any ball such that $$\lambda_\mathcal{H}(E_n)=d_\mathcal{H}(E_n,B_n)=\inf_{|B|=\pi}d_\mathcal{H}(E_n,B).$$ Similarly, we denote by $B_0$ a ball such that $$\lambda_\mathcal{H}(E_0)=d_\mathcal{H}(E_0,B_0)=\inf_{|B|=\pi}d_\mathcal{H}(E_0,B).$$

Using the definition of $\lambda_\mathcal{H}$ and the triangle inequality, we obtain that 
$$\lambda_\mathcal{H}(E_n)\leq d_\mathcal{H}(E_n,B_0)\leq d_\mathcal{H}(E_0,B_0)+d_\mathcal{H}(E_n,E_0)$$ $$=d_\mathcal{H}(E_n,E_0)+\lambda_\mathcal{H}(E_0)\qquad\qquad \forall\ n\in \mathbb{N}.$$
Therefore for all $n\in \mathbb{N}$, it results
\begin{equation}
\label{eq:7}
\lambda_\mathcal{H}(E_n)\leq d_\mathcal{H}(E_n,E_0)+\lambda_\mathcal{H}(E_0).
\end{equation}
Passing to the limit as $n\to \infty$, we have $$\limsup_{n} \lambda_\mathcal{H}(E_n)\leq \lambda_\mathcal{H}(E_0),$$ that is, $\lambda_\mathcal{H}$ is upper semi-continuous.
\end{proof}

Now we are ready to prove the following theorem:
\begin{thm}
There exists $\overline{R}>0$ such that for all $R<\overline{R}$ there exists a set $E_0\in \mathcal{K}^1_R$ which minimizes the shape functional $$\mathcal{F}(E)=\frac{D(E)}{\lambda_\mathcal{H}^2(E)}$$ among all the subsets of $\mathbb{R}^2$ belonging to $\mathcal{K}_R^1$ (the ball excluded).
\end{thm}
\begin{proof} The proof is divided into two steps.

\textbf{Step 1.} In this step we prove that the functional $\mathcal{F}$ is bounded from below away from zero.

For $n=2$ Theorem \ref{teo_1.1} states that for any $R>0$ there exist $0<\delta_R<1$ and a costant $C_R$ depending only on $R$ and $n$ such that $$\frac{D(E)}{\lambda^2_\mathcal{H}(E)}\geq C_R,\qquad  \textit{if}\ \ D(E)\leq \delta_R.$$
If $E\in \mathcal{K}_R^1$ is such that $D(E)\geq\delta_R$, then $$\frac{D(E)}{\lambda^2_\mathcal{H}(E)}\geq \frac{\delta_R}{\lambda^2_\mathcal{H}(E)}.$$ To obtain an estimate from below of $\delta_R/\lambda_\mathcal{H}^2$ we only need to obtain an estimate from above of $\lambda_\mathcal{H}$. Considering $x_1,x_2\in E$, by the convexity of $E$ we have that $\overline{x_1x_2}\subset E$. We can now consider the cones with vertex in $x_1$ and in $x_2$ and radius $R$. If we consider in these two cones two balls with center $x_1^c$ and $x_2^c$ rispectively and radius $\bar{R}\ll R$, it results $$\bar{R}\leq d(x_1,x_1^c)\leq R-\bar{R}$$ $$\bar{R}\leq d(x_2,x_2^c)\leq R-\bar{R}$$
and so the area of the rectangle with basis $d(x_1^c,x_2^c)$ and height $2\bar{R}$ is such that $$d(x_1^c,x_2^c)\leq \frac{\pi}{2\bar{R}}.$$ 
By the triangle inequality, it holds $$d(x_1,x_2)\leq d(x_1,x_1^c)+d(x_1^c,x_2^c)+d(x_2^c,x_2)\leq 2R-2\bar{R}+\frac{\pi}{2\bar{R}}=H_R.$$
We observe that if $E\cap B\neq \emptyset$, where we denote by $B$ an unit ball, then $$d_\mathcal{H}(E,B)\leq \max\{diam (E),diam (B)\}.$$ By the isodiametric inequality (see \cite[Theorem 3.11]{maggi_2012}), it holds $$diam(B)\leq diam (E)$$ and so
\begin{equation}
\label{eq:4}
\lambda_\mathcal{H}(E)=d_\mathcal{H}(E,B_1(x_\infty))\leq diam(E)\leq H_R,
\end{equation}
where we denote by $B_1(x_\infty)$ any ball such that $\lambda_\mathcal{H}(E)=d_\mathcal{H}(E,B_1(x_\infty)).$ 
Consequentely
$$\frac{D(E)}{\lambda^2_\mathcal{H}(E)}\geq \frac{\delta_R}{H_R^2}.$$
As a consequence we have for every $R>0$
$$\frac{D(E)}{\lambda^2_\mathcal{H}(E)}\geq \overline{C}_R=\min \left\{C_R,\frac{\delta_R}{H^2_R}\right\}\qquad \forall\ E\in \mathcal{K}_R^1.$$
\textbf{Step 2.} Now we prove the existence of an optimal set in the class $\mathcal{K}_R^1$ for every $R>0$.
Let $E_n$ be a minimizing sequence in the class $\mathcal{K}_R^1$, then 
\begin{equation}
\label{eq:2}
\inf_{E\in \mathcal{K}_R^1} \mathcal{F}(E)= \lim_n \frac{D(E_n)}{\lambda_\mathcal{H}^2(E_n)}\leq C.
\end{equation}
We observe that exists $C>0$ such that $\lambda_\mathcal{H}(E_n)\leq C$. In fact if $D(E_n)\leq \delta_R$, then by Theorem \ref{teo_1.1} $\lambda_\mathcal{H}^2(E_n)\leq C$.

If $D(E_n)\geq \delta_R$ then, by \eqref{eq:4}, $\lambda_\mathcal{H}(E_n)$ is uniformly bounded.

Therefore by \eqref{eq:2} there exists a constant $\overline{C}$ such that $D(E_n)\leq \overline{C}$ and so $P(E_n)\leq \overline{C}$. By the compactness of the Hausdorff distance on equibounded sets, it exists a set $E_0\subset\mathbb{R}^2$ such that

\begin{equation}
\label{eq:5}
E_n\xrightarrow{\mathcal{H}} E_0.
\end{equation}

By Lemma \ref{lem:1} $E_0\in \mathcal{K}_R^1$.

There are two possibilities
\begin{enumerate}
\item[a)] $\qquad E_0$ is a ball
\item[b)] $\qquad E_0$ is not a ball
\end{enumerate}
\begin{enumerate}
\item[a)] In this case we can apply the Fuglede's results to the set $E_0$. More precisely it was proved in \cite[Theorem 2.3]{article} the existence of a positive constant $C_F$ which depends only on $n$ such that $$\frac{D(E_0)}{\lambda_\mathcal{H}^2(E_0)}\geq C_F>0.$$ 

Then exists $n_0\in \mathbb{N}$ such that for all $n>n_0$ it holds $$\frac{D(E_n)}{\lambda_\mathcal{H}^2(E_n)}\geq C_F>0.$$

Since \eqref{eq:5} holds with $E_0=B$, it results that $\lambda_\mathcal{H}(E_n)\rightarrow 0$ and by \eqref{eq:2} it results that $ D(E_n)\rightarrow 0$.

To exclude the case $a)$ we claim that exists a set $\tilde{E}\in \mathcal{K}_R^1$ such that $$\frac{D(\tilde{E})}{\lambda_\mathcal{H}^2(\tilde{E})}<C_F.$$
To this aim we consider a rectangle $Q_R$ with dimensions $2L>2l$ such that $2L2l=\pi$. This set belongs to $\mathcal{K}_R^1$ with $R=2l<2$. To compute $\lambda_\mathcal{H}$ we compute the distance between one vertex $A$ to the unit ball $B$ centered in the center of the rectangle $$d(A,B)=\sqrt{L^2+l^2}-1=\sqrt{\frac{\pi^2}{16l^2}+l^2}-1,$$ 
$$\frac{D(Q_R)}{\lambda_\mathcal{H}^2(Q_R)}=\frac{(4(l+L)-2\pi)/2\pi}{\max^2\left\{1-l,\sqrt{\frac{\pi^2}{16^2}+l^2}-1\right\}}$$ $$=\frac{2\left(\frac{\pi}{4l}+l\right)-\pi}{\pi\max^2\left\{1-l,\sqrt{\frac{\pi^2}{16^2}+l^2}-1\right\}}=a(l).$$
We observe that $a(l)$ is strictly increasing in $(0,1)$ and in this case there exists $R_0<2$ such that for $R<R_0$ the rectangle $Q_R\in \mathcal{K}_R^1$ satisfies $$\frac{D(Q_R)}{\lambda_\mathcal{H}^2(Q_R)}<C_F.$$
\item[b)] In this other case we consider that the limit $E_0$ is not a ball. 
By the upper semi-continuity of $\lambda_\mathcal{H}$ (Lemma \ref{lem:2}) and by the lower semicontinuity of the perimeter we obtain that $$\frac{D(E_0)}{\lambda_\mathcal{H}^2(E_0)}\leq \liminf_{n}\frac{D(E_n)}{\lambda_\mathcal{H}^2(E_n)}.$$
By \eqref{eq:7} it results
$$\frac{D(E_n)}{\lambda_\mathcal{H}^2(E_n)}\geq \frac{D(E_n)}{(d_\mathcal{H}(E_n,E_0)+\lambda_\mathcal{H}(E_0))^2}$$ and passing to the limit we obtain that
$$\inf_{E\in \mathcal{K}^1_R}\frac{D(E)}{\lambda_\mathcal{H}^2(E)}=\lim_n \frac{D(E_n)}{\lambda_\mathcal{H}^2(E_n)}\geq \liminf_n \frac{D(E_n)}{(d_\mathcal{H}(E_n,E_0)+\lambda_\mathcal{H}(E_0))^2}$$ $$\geq \frac{D(E_0)}{\lambda_\mathcal{H}^2(E_0)}\geq \inf_{E\in \mathcal{K}^1_R}\frac{D(E)}{\lambda_\mathcal{H}^2(E)}.$$

 Therefore $$\frac{D(E_0)}{\lambda_\mathcal{H}^2(E_0)}= \inf_{E\in \mathcal{K}^1_R}\frac{D(E)}{\lambda_\mathcal{H}^2(E)}$$ and so $E_0$ minimizes the functional $\mathcal{F}(E)$.

\end{enumerate}

\end{proof}

\bibliographystyle{plain}

\end{document}